\documentclass{amsart}

\usepackage{macros}
\standardsettings

\draftfalse

\begin{document}
\title{Exact dimensions of the prime continued fraction Cantor set}

\authortushar\authordavid

%\subjclass[2010]{Primary }
%\keywords{}
%\date{}
%\dedicatory{}

\begin{abstract}
We study the exact Hausdorff and packing dimensions of the \textsf{prime Cantor set}, $\Lambda_P$, which comprises the irrationals whose continued fraction entries are prime numbers. 
%The problem was first raised during the end of the last century following seminal work of Mauldin and Urba\'nski, who leveraged Erd{\doubleacute o}s' results on the distribution of two-sided prime gaps to show that the Hausdorff and packing measures of the prime Cantor set were zero and infinity respectively. In sharp contrast, determining the exact dimensions of the prime Cantor set has proven to be a surprisingly resistant problem. 
We prove that the Hausdorff measure of the prime Cantor set cannot be finite and positive with respect to any sufficiently regular dimension function, thus negatively answering a question of Mauldin (2013) for this class of dimension functions. By contrast, under a reasonable number-theoretic conjecture we prove that the packing measure of the conformal measure on the prime Cantor set is in fact positive and finite with respect to the dimension function $\psi(r) = r^\delta \log^{-2\delta}\log(1/r)$, where $\delta$ is the dimension (conformal, Hausdorff, and packing) of the prime Cantor set.
\end{abstract}

\maketitle

Let $\Lambda_P$ denote the \textsf{prime Cantor set}, i.e. the Cantor set of irrationals whose continued fraction entries are prime numbers, and let $\delta = \delta_P$ denote the common value \cite[Theorems 2.7 and 2.11]{MauldinUrbanski4} for the Hausdorff and packing dimensions of $\Lambda_P$. Using a result due to Erd{\doubleacute o}s \cite{Erdos} that guarantees the existence of arbitrarily large two-sided gaps in the sequence of primes, Mauldin and Urba\'nski \cite[Corollary 4.5 and Corollary 5.6]{MauldinUrbanski4} proved that despite there being a conformal measure and a corresponding invariant Borel probability measure for this conformal iterated function system, the $\delta$-dimensional Hausdorff and packing measures were zero and infinity respectively. 
This led naturally to the surprisingly resistant problem\footnote{This problem was the second of a trio concerning the fine geometric properties of continued fraction Cantor sets posed by Mauldin and Urba\'nski  \cite[\67]{MauldinUrbanski4}. The first problem (that asked for the existence of a continued fraction Cantor set with positive and finite Hausdorff and packing measures) was repeated in the 2013 Erd{\doubleacute o}s centennial volume \cite[Problem 7.2]{Mauldin}, and resolved in 2014 by Simmons--Fishman--Urba\'nski \cite[Theorem 7.1]{FSU1}. The third problem (that asked for the existence of a continued fraction Cantor set of Hausdorff dimension $t$ for every $t \in (0,1]$) was resolved in 2006 by Kesseb\"ohmer--Zhu, \cite{KessebohmerZhu}.} (\cite[\67]{MauldinUrbanski4} and \cite[Problem 7.1]{Mauldin}) of determining whether there was an appropriate dimension function with respect to which the Hausdorff and packing measures of  $\Lambda_P$ were positive and finite.

\section{Main theorems}

We start with stating our main results, precise definitions will follow in the next section. Let $\delta = \delta_P$ denote the common value \cite[Theorems 2.7 and 2.11]{MauldinUrbanski4} for the Hausdorff and packing dimensions of $\Lambda_P$. If $\mu$ is a locally finite Borel measure on $\R$, then we let
\begin{align*}
\HH^\psi(\mu) &\df \inf\left\{\HH^\psi(A): \mu(\R\butnot A) = 0\right\},\\
\PP^\psi(\mu) &\df \inf\left\{\PP^\psi(A): \mu(\R\butnot A) = 0\right\}.
\end{align*}

\begin{theorem}
\label{theoremHD}
Let $\mu = \mu_P$ be the conformal measure on $\Lambda_P$, and let $\psi$ be a doubling\Footnote{A function $\psi$ is \textsf{doubling} if for all $C_1 \geq 1$, there exists $C_2 \geq 1$ such that for all $x,y$ with $C_1^{-1} \leq x/y \leq C_1$, we have $C_2^{-1} \leq \psi(x)/\psi(y) \leq C_2$.} dimension function such that $\Psi(r) = r^{-\delta} \psi(r)$ is monotonic.
%and let $\Psi(r) = r^{-\delta} \psi(r)$. Suppose that $\Psi$ is monotonic. 
Then $\HH^\psi(\mu) = 0$ if the series
\begin{equation}
\label{HDseries}
\sum_{k=1}^\infty \frac{y^{\frac{1-2\delta}{1-\delta}}}{(\log y)^{\frac{\delta}{1-\delta}}}\given_{y=\Psi(\lambda^{-k})}
\end{equation}
diverges, and $=\infty$ if it converges, for all (equiv. for any) fixed $\lambda > 1$.
\end{theorem}
\noindent Note that $1/2 < \delta \approx 0.657 < 1$ \cite[Table 1 and Section 3]{CLU2}, so the exponent in the numerator is negative.

The following corollary negatively resolves \cite[Problem 7.1]{Mauldin} for sufficiently regular dimension functions, for example Hardy $L$-functions \cite{Hardy,Hardy2}.

\begin{corollary}
For any doubling dimension function $\psi$ such that $\Psi(r) = r^{-\delta} \psi(r)$ is monotonic, we have $\HH^\psi(\Lambda_P) \in \{0,\infty\}$.
\end{corollary}
\begin{proof}
By way of contradiction suppose that $0 < \HH^\psi(\Lambda_P) < \infty$. Then $\HH^\psi\given \Lambda_P$ is a conformal measure on $\Lambda_P$ and therefore a scalar multiple of $\mu_P$, and thus $\HH^\psi(\Lambda_P) = \HH^\psi(\HH^\psi\given \Lambda_P) \in \{0,\infty\}$ by Theorem \ref{theoremHD}, a contradiction.
\end{proof}

\begin{remark*}
Letting $\psi(r) = r^\delta \log^s(1/r)$ with $s > \frac{1-\delta}{2\delta-1}$ gives an example of a function that satisfies the hypotheses of Theorem \ref{theoremHD} such that the series \eqref{HDseries} converges. For this function, we have $\HH^\psi(\Lambda_P) \geq \HH^\psi(\mu) = \infty$. This affirmatively answers \cite[Problem 2 in Section 7]{MauldinUrbanski4}.
\end{remark*}

\begin{theorem}
\label{theoremPD1}
Let $\mu$ be the conformal measure on $\Lambda_P$, let $\theta = 21/40$, and let
\begin{equation}
\label{phidef}
\phi(x) = \frac{\log(x)\log\log(x)\log\log\log\log(x)}{\log^2\log\log(x)}\cdot
\end{equation}
Then
\begin{align}
\PP^\psi(\mu) &= \infty \text{ where } \psi(r) = r^\delta \phi^{-\delta}(\log(1/r)). \label{Ppsimuinfty}\\
\PP^\psi(\mu) &= 0 \text{ where } \psi(r) = r^\delta \log^{-s}(1/r) \text{ if } s > \theta\delta/(2\delta - 1) \label{Ppsimu0}
\end{align}
\end{theorem}

We can get a stronger result for packing measure by assuming the following:

\begin{conjecture}
\label{conjecturecramergranville}
Let $p_n$ denote the $n$th prime, and let $d_n = p_{n + 1} - p_n$. For each $k\geq 1$ let
\begin{align*}
R_k &\df \limsup_{n\to\infty} \frac{\min(d_{n+1},\ldots,d_{n + k})}{\log^2(p_n)}\cdot
\end{align*}
Then $0 < R_k < \infty$ for all $k\in\N$.
\end{conjecture}

\begin{remark*}
The case $k = 1$ of Conjecture \ref{conjecturecramergranville} is known as the Cram\'er--Granville conjecture. Early heuristics led Harald Cram\'er to conjecture that it is true with $R_1 = 1$ \cite{Cramer}, but improved heuristics now suggest that $R_1 = 2e^{-\gamma}$, where $\gamma$ is the Euler--Mascheroni constant; see \cite{Granville, Granville2,Pintz}. Applying Cram\'er's heuristics to the case $k\geq 2$ of Conjecture \ref{conjecturecramergranville} yields the prediction that $R_k = 1/k$,\Footnote{Specifically, assume that each integer $n$ has probability $\frac1{\log(n)}$ of being prime. Under this assumption, for $m\leq n$ the probability that no integers in an interval $\OC{n}{n+m}$ are prime is approximately $(1 - \frac1{\log(n)})^m \asymp \exp(-\frac{m}{\log(n)})$. Thus, the probability that $d_n > m$ is approximately $\exp(-\frac{m}{\log(p_n)})$, since $d_n > m$ if and only if the interval $\OC{p_n}{p_n+m}$ has no primes. So the probability that $\min(d_{n+1},\ldots,d_{n+k}) > m$ is approximately $\exp(-\frac{km}{\log(p_n)})$. Now fix a constant $C > 0$. The probability that $\min(d_{n+1},\ldots,d_{n+k}) \geq C\log^2(p_n)$ is approximately $\exp(-\frac{kC\log^2(p_n)}{\log(p_n)}) = p_n^{-kC}$. Now, by the Borel--Cantelli lemma, the probability is 1 that $\min(d_{n+1},\ldots,d_{n+k}) \geq C\log^2(p_n)$ for infinitely many $n$ if and only if the series $\sum_n p_n^{-kC}$ diverges, which by the prime number theorem is true if and only if $C \leq 1/k$. It follows (under this probabilistic model) that $R_k = 1/k$, where $R_k$ is as in Conjecture \ref{conjecturecramergranville}.} so perhaps an appropriate correction would be $R_k = 2 e^{-\gamma}/k$.
% Applying these same heuristics to the case of larger $k$ yields the conjecture $R_k = 2 e^{-\gamma}/k$ (see Appendix\internal).
\end{remark*}

\begin{theorem}
\label{theoremPD2}
If the cases $k=1,2$ of Conjecture \ref{conjecturecramergranville} are true, then $\PP^\psi(\mu) \in (0,\infty)$, where $\psi$ is given by the formula
\begin{equation}
\label{psidef}
\psi(r) = r^\delta \log^{-2\delta}\log(1/r),
\end{equation}
\end{theorem}

Note that the cases $k=1,2$ of Conjecture \ref{conjecturecramergranville} correspond to information about the lengths of one-sided and two-sided gaps in the primes, respectively.

\section{Preliminaries and notation}

\begin{convention}
In what follows, $A \lesssim B$ means that there exists a constant $C > 0$ such that $A \leq C B$. $A\asymp B$ means $A \lesssim B \lesssim A$. $A \lesssim_\plus B$ means there exists a constant $C$ such that $A \leq B + C$. $A \lesssim_{\plus,\times} B$ means that there exist constants $C_1,C_2$ such that $A \leq C_1 B + C_2$. 
\end{convention}

\begin{convention}
All measures and sets are assumed to be Borel, and measures are assumed to be locally finite. Sometimes we restate these hypotheses for emphasis.
\end{convention}

Recall that the \textsf{continued fraction expansion} of an irrational number $x\in (0,1)$ is the unique sequence of positive integers $(a_n)$ such that
\[
x = [0;a_1,a_2,\ldots] \df \cfrac{1}{a_1 + \cfrac{1}{a_2 + \ddots}}
\]
Given $E \subset \N$, we define the set $\Lambda_E$ to be the set of all irrationals in $(0,1)$ whose continued fraction expansions lie entirely in $E$. Equivalently, $\Lambda_E$ is the image of $E^\N$ under the \textsf{coding map} $\pi:\N^\N \to (0,1)$ defined by $\pi((a_n)) = [0;a_1,a_2,\ldots]$.

The set $\Lambda_E$ can be studied dynamically in terms of its corresponding \textsf{Gauss iterated function system}, i.e. the collection of maps $\Phi_E \df (\phi_a)_{a\in E}$, where
\[
\phi_a(x) \df \frac{1}{a + x}\cdot
\]
(The Gauss IFS $\Phi_E$ is a special case of a \textsf{conformal iterated function system} (see e.g. \cite{MauldinUrbanski1,MauldinUrbanski4,CLU2}), but in this paper we deal only with the Gauss IFS case.) Let $E^* = \bigcup_{n\geq 0} E^n$ denote the collection of finite words in the alphabet $E$. For each $\omega\in E^*$, let $\phi_\omega = \phi_{\omega_1}\circ\cdots \circ \phi_{\omega_{|\omega|}}$, where $|\omega|$ denotes the length of $\omega$. Then
\[
\pi(\omega) = \lim_{n\to\infty} \phi_{\omega\given[1,n]}(0).
\]
Equivalently, $\pi(\omega)$ is the unique intersection point of the \textsf{cylinder sets} $[\omega\given [1,n]]$, where
\[
[\omega] \df \phi_\omega([0,1]).
\]
Next, we define the \textsf{pressure} of a real number $s \geq 0$ to be
\[
\P_E(s) \df \lim_{n\to\infty} \frac1n \log \sum_{\omega\in E^n} \|\phi_\omega'\|^s
\]
where $\|\phi_\omega'\| \df \sup_{x\in [0,1]} |\phi_\omega'(x)|$. The Gauss IFS $\Phi_E$ is called \textsf{regular} if there exists $\delta = \delta_E \geq 0$ such that $\P_E(\delta_E) = 0$. The following result was proven in \cite{MauldinUrbanski1}.

\begin{proposition}[{\cite[Theorem 3.5]{MauldinUrbanski1}}]
\label{propositionexistsconformal}
Let $\Phi_E$ be a regular (Gauss) IFS. Then there exists a unique measure $\mu = \mu_E$ on $\Lambda_E$ such that
\[
\mu_E(A) = \sum_{a\in E} \int_{\phi_a^{-1}(A)} |\phi_a'(x)|^{\delta_E} \;\dee\mu_E(x)
\]
for all $A \subset [0,1]$.
\end{proposition}

The measure $\mu$ appearing in Proposition \ref{propositionexistsconformal} is called the \textsf{conformal measure} of $\Phi_E$, and $\delta_E$ is called the \textsf{conformal dimension} of $\Phi_E$. Recall that the \textsf{bounded distortion property} (cf. \cite[(2.9)]{MauldinUrbanski1}) states that
\[
|\phi_\omega'(x)| \asymp \|\phi_\omega'\| \text{ for all } \omega\in E^* \text{ and } x\in [0,1].
\]
This implies that the measure of a cylinder set $[\omega]$ satisfies
\[
\mu(\omega) \df \mu([\omega]) \asymp \|\phi_\omega'\|^\delta
\]
and that
\begin{equation}
\label{boundeddistortion}
\mu(\omega\tau) \asymp \mu(\omega)\mu(\tau) \;\;\text{ for all $\omega,\tau\in E^*$.}
\end{equation}

\begin{convention}
For shorthand we write $\mu(A) = \sum_{\omega\in A} \mu(\omega)$ for all $A \subset E^*$, and $\mu(A) = \mu(\pi(A))$ for all $A \subset E^\N$.
\end{convention}

The aim of this paper is to study the Hausdorff and packing measures of the measure $\mu_P$, where $P\subset\N$ is the set of primes. To define these quantities, let $\psi:(0,\infty) \to (0,\infty)$ be a \textsf{dimension function}, i.e. a continuous increasing function such that $\lim_{r\to 0} \psi(r) = 0$. Then the \textsf{$\psi$-dimensional Hausdorff measure} of a set $A\subset\R$ is
\[
\HH^\psi(A) \df \lim_{\epsilon\searrow 0}\inf\left\{\sum_{i = 1}^\infty \psi(\diam(U_i)): \text{$(U_i)_1^\infty$ is a countable cover of $A$ with $\diam(U_i)\leq\epsilon \all i$}\right\}
\]
and the  \textsf{$\psi$-dimensional packing measure} of $A$ is defined by the formulas
\[
\w\PP^\psi(A) \df \lim_{\epsilon\searrow 0}\sup\left\{\sum_{j = 1}^\infty \psi(\diam(B_j)):
\begin{split}
&\text{$(B_j)_1^\infty$ is a countable disjoint collection of balls}\\
&\text{with centers in $A$ and with $\diam(B_j)\leq\epsilon \all j$}
\end{split}\right\}
\]
\[
\PP^\psi(A) \df \inf\left\{ \sum_{i = 1}^\infty \w\PP^\psi(A_i) : A \subset \bigcup_{i = 1}^\infty A_i\right\}.
\]
A special case is when $\psi(r) = r^s$ for some $s > 0$, in which case the shorthands $\HH^\psi = \HH^s$ and $\PP^\psi = \PP^s$ are used.

If $\mu$ is a locally finite Borel measure on $\R$, then we let
\begin{align*}
\HH^\psi(\mu) &\df \inf\left\{\HH^\psi(A): \mu(\R\butnot A) = 0\right\},\\
\PP^\psi(\mu) &\df \inf\left\{\PP^\psi(A): \mu(\R\butnot A) = 0\right\}.
\end{align*}
This is analogous to the definitions of the (upper) Hausdorff and packing dimensions of $\mu$, see \cite[Proposition 10.3]{Falconer_book3}.
\begin{remark*}
The Hausdorff and packing dimensions of sets \cite[Section 2.1]{Falconer_book3} and the (upper) Hausdorff and packing dimensions of measures \cite[Proposition 10.3]{Falconer_book3} can be defined in terms of $\HH^s$ and $\PP^s$ as follows:
\begin{align*}
\HD(A) &\df \sup\{s\geq 0:\HH^s(A) > 0\},&
\overline\HD(\mu) &\df \sup\{s\geq 0:\HH^s(\mu) > 0\},\\
\PD(A) &\df \sup\{s\geq 0:\PP^s(A) > 0\},&
\overline\PD(\mu) &\df \sup\{s\geq 0:\PP^s(\mu) > 0\}.
\end{align*}
It follows from \cite[Theorems 2.7 and 2.11]{MauldinUrbanski4} and Theorems \ref{theoremHD} and \ref{theoremPD1} above that
\[
\HD(\Lambda_P) = \PD(\Lambda_P) = \overline\HD(\mu_P) = \overline\PD(\mu_P) = \delta_P.
\]
\end{remark*}
For each point $x\in\R$ let
\begin{align*}
\overline D_\mu^\psi(x) &\df \limsup_{r\searrow 0} \frac{\mu(B(x,r))}{\psi(r)},\\
\underline D_\mu^\psi(x) &\df \liminf_{r\searrow 0} \frac{\mu(B(x,r))}{\psi(r)}\cdot
\end{align*}

\begin{theorem}[Rogers--Taylor--Tricot density theorem, {\cite[Theorems 2.1 and 5.4]{TaylorTricot}}, see also \cite{RogersTaylor}]
\label{theoremRTT}
Let $\mu$ be a positive and finite Borel measure on $\R$, and let $\psi$ be a dimension function. Then for every Borel set $A\subset\R$,
\begin{align} \label{rogerstaylor}
\mu(A) \inf_{x\in A}\frac{1}{\overline D_\mu^\psi(x)} \lesssim_\times \HH^\psi(A) &\lesssim_\times \mu(\R) \sup_{x\in A}\frac{1}{\overline D_\mu^\psi(x)}\\ \label{taylortricot}
\mu(A) \inf_{x\in A}\frac{1}{\underline D_\mu^\psi(x)} \lesssim_\times \PP^\psi(A) &\lesssim_\times \mu(\R) \sup_{x\in A}\frac{1}{\underline D_\mu^\psi(x)}\cdot
\end{align}
\end{theorem}

\begin{corollary}
\label{corollaryRTT}
Let $\mu,\psi$ be as in Theorem \ref{theoremRTT}. Then
\begin{align} \label{rogerstaylor2}
\HH^\psi(\mu) &\asymp_\times \esssup_{x\sim \mu}\frac{1}{\overline D_\mu^\psi(x)}\\ \label{taylortricot2}
\PP^\psi(\mu) &\asymp_\times \esssup_{x\sim \mu}\frac{1}{\underline D_\mu^\psi(x)}\cdot
\end{align}
Here the implied constants may depend on $\mu$ and $\psi$, and $\esssup_{x\sim\mu}$ denotes the essential supremum with $x$ distributed according to $\mu$.
\end{corollary}
\begin{proof}
We prove \eqref{rogerstaylor2}; \eqref{taylortricot2} is similar. For the $\lesssim$ direction, take
\[
A = \left\{x : \tfrac{1}{\overline D_\mu^\psi(x)} \leq \esssup_{y\sim \mu}\tfrac{1}{\overline D_\mu^\psi(y)}\right\}
\]
in the right half of \eqref{rogerstaylor}. $A$ has full $\mu$-measure, so $\HH^\psi(\mu) \leq \HH^\psi(A)$. For the $\gtrsim$ direction, let $B$ be a set of full $\mu$-measure, fix $t < \esssup_{y\sim \mu}\tfrac{1}{\overline D_\mu^\psi(y)}$, and let
\[
A = B\cap \left\{x : \tfrac{1}{\overline D_\mu^\psi(x)} \geq t \right\}.
\]
Then $\mu(A) > 0$. Applying the left half of \eqref{rogerstaylor}, using $\HH^\psi(A) \leq \HH^\psi(B)$, and then taking the infimum over all $B$ and supremum over $t$ yields the $\gtrsim$ direction of \eqref{rogerstaylor2}.
\end{proof}

\begin{remark*}
For a doubling dimension function $\psi$ and a conformal measure $\mu = \mu_E$, the $\esssup$ in \eqref{rogerstaylor2}-\eqref{taylortricot2} can be replacd by $\essinf$ due to the ergodicity of the shift map $\sigma$ with respect to $\mu$ \cite[Theorem 3.8]{MauldinUrbanski1}. Indeed, a routine calculation shows that $\overline D_\mu^\psi(x) \asymp \overline D_\mu^\psi(\sigma(x))$ for all $x$, whence ergodicity implies that the function $x\mapsto\overline D_\mu^\psi(x)$ is constant $\mu$-a.e., and similarly for $x\mapsto \underline D_\mu^\psi(x)$.
\end{remark*}

{\bf Terminological note.} If $\psi$ is a dimension function such that $\HH^\psi(A)$ (resp. $\HH^\psi(\mu)$) is positive and finite, then $\psi$ is called an \textsf{exact Hausdorff dimension function} for $A$ (resp. $\mu$). Similar terminology applies to packing dimension.

%More generally, the functional $\psi \mapsto \HH^\psi(A)$ (resp. $\psi\mapsto \HH^\psi(\mu)$ is called the \textsf{exact Hausdorff dimension} of $A$ (resp. $\mu$).

\section{Results for regular Gauss IFSes}

In this section we consider a regular Gauss IFS $\Phi_E$ and state some results concerning $\HH^\psi(\mu_E)$ and $\PP^\psi(\mu_E)$, given appropriate assumptions on $E$ and $\psi$. Throughout the section we will make use of the following assumptions, all of which hold for the prime Gauss IFS $\Phi_P$:

\begin{assumption}
\label{assumptionPNT}
The set $E \subset \N$ satisfies an asymptotic law
\begin{equation}
\label{EN2N}
\#(E\cap [N,2N]) \asymp f(N),
\end{equation}
where $f$ is regularly varying with exponent $s\in (\delta,2\delta)$.\Footnote{A function $f$ is said to be \textsf{regularly varying with exponent $s$} if for all $a > 1$, we have $\lim_{x\to\infty} \frac{f(ax)}{f(x)} = a^s$.} For example, if $E$ is the set of primes, then by the prime number theorem $f(N) = N/\log(N)$ satisfies \eqref{EN2N}, and $f$ is regularly varying with exponent $s = 1 \in (\delta,2\delta)$, since $1/2 < \delta_P < 1$.
\end{assumption}

\begin{assumption}
\label{assumptionC}
There exists $\lambda > 1$ such that for all $0 < r \leq 1$,
\[
\mu(\{a\in E : \lambda^{-1} r < \|\phi_a'\| \leq r\}) \asymp \mu(\{a\in E : \|\phi_a'\| \leq r\}).
\]
For example, if $E$ is the set of primes, then this assumption follows from the prime number theorem via a routine calculation showing that both sides are $\asymp \frac{r^\delta}{\log(1/r)}$.
\end{assumption}

\begin{assumption}
\label{assumptionlyapunov}
The Lyapunov exponent $-\sum_{a\in E} \mu(a) \log\|\phi_a'\|$ is finite. Note that this is satisfied when $E$ is the set of primes, since $\mu(a) \log\|\phi_a'\| \asymp a^{-2\delta} \log(a)$ and $\delta > 1/2$.
% \cite[Corollary 4.5]{MauldinUrbanski4} says primes are strongly regular
% \cite[Proposition 3.13]{FSU_Gauss} says strongly regular implies finite Lyapunov exponent
\end{assumption}

For each $k\in\N$, let
\begin{equation}
\label{Wkdef}
W_k \df \{\omega\in E^* : \lambda^{-(k + 1)} < \|\phi_\omega'\| \leq \lambda^{-k}% < \|\phi_{\omega}\given |\omega|-1}'\|
\}.
\end{equation}
Note that although the sets $([\omega])_{\omega\in W_k}$ are not necessarily disjoint, there is a uniform bound (depending on $\lambda$) on the multiplicity of the collection, i.e. there exists a constant $C$ independent of $k$ such that $\sup_x\{\#\{\omega\in W_k : x\in [\omega]\}\} \leq C$.

\begin{lemma}
\label{lemmamuJk}
Assume that Assumption \ref{assumptionC} holds. Let $\JJ = (J_k)_1^\infty$ be a sequence of subsets of $E$, and let
\begin{align*}
\Sigma_\JJ &\df \sum_{k = 1}^\infty \mu(J_k)\\
S_\JJ &\df \{\omega \in E^\N : \text{ there exist infinitely many $(n,k)$ such that } \omega\given n \in W_k, \; \omega_{n + 1} \in J_k\}.
\end{align*}
Then $\mu(S_\JJ) > 0$ if $\Sigma_\JJ = \infty$, and $\mu(S_\JJ)=0$ otherwise.
\end{lemma}
\begin{proof}
For each $k\in \N$, let
\[
A_k = \bigcup\Big\{[\omega a] : \omega\in W_k,\; a\in J_k\Big\}.
\]
We claim that
\begin{itemize}
\item[1.] $\mu(A_k) \asymp \mu(J_k)$ and that
\item[2.] the sequence $(A_k)_1^\infty$ is quasi-independent, meaning that $\mu(A_k \cap A_\ell) \lesssim \mu(A_k) \mu(A_\ell)$ whenever $k\neq \ell$.
\end{itemize}
\begin{subproof}[Proof of 1]
Since the collection $([\omega])_{\omega\in W_k}$ has bounded multiplicity, we have
\[
\mu(A_k) \asymp \sum_{\omega\in W_k} \sum_{a\in J_k} \mu(\omega a) \asymp \sum_{\omega\in W_k} \mu(\omega)\mu(J_k)
\]
and
\begin{align*}
\sum_{\omega\in W_k} \mu(\omega)
&\asymp \sum_{\substack{\omega \in W_k \\ \|\phi_{\omega\given|\omega| - 1}'\| > \lambda^{-k}}} \mu(\omega)
\since{$([\omega])_{\omega\in W_k}$ has bounded multiplicity}\\
&\asymp \sum_{\substack{\omega\in E^* \\ \|\phi_\omega'\| > \lambda^{-k}}} \sum_{\substack{a\in E \\ \omega a \in W_k}} \mu(\omega) \mu(a)\\
&\asymp \sum_{\substack{\omega\in E^* \\ \|\phi_\omega'\| > \lambda^{-k}}} \mu(\omega)  \sum_{\substack{a\in E \\ \|\phi_{\omega a}'\| \leq \lambda^{-k}}} \mu(a) \by{Assumption \ref{assumptionC}}\\\
&\asymp \sum_{\substack{\omega\in E^* \\ \|\phi_\omega'\| > \lambda^{-k}}} \sum_{\substack{a\in E \\ \|\phi_{\omega a}'\|\leq \lambda^{-k}}} \mu(\omega a) = \mu([0,1]) = 1.
\end{align*}
\end{subproof}
\begin{subproof}[Proof of 2]
Let $k < \ell$. Then
\begin{align*}
\mu(A_k\cap A_\ell) &= \sum_{\omega\in W_k} \sum_{a\in J_k} \sum_{\substack{\tau\in E^* \\ \omega a \tau \in W_\ell}} \sum_{b\in J_\ell} \mu(\omega a \tau b)\\
&\asymp \sum_{\omega\in W_k} \mu(\omega) \sum_{a\in J_k} \mu(a) \sum_{\substack{\tau\in E^* \\ \omega a \tau \in W_\ell}} \mu(\tau) \mu(J_\ell)\\
&\lesssim \mu(J_k) \mu(J_\ell) \asymp \mu(A_k) \mu(A_\ell),
\end{align*}
where the $\lesssim$ in the last line is because the collection $\{[\tau] : \tau\in E^*, \omega a \tau \in W_\ell\}$ has bounded multiplicity.
\end{subproof}
Applying the Borel--Cantelli lemma (see e.g. \cite{BeresnevichVelani7}) completes the proof.
\end{proof}

\begin{theorem}[Global measure formula for Gauss IFSes]
\label{theoremGMF}
Let $\Phi_E$ be a regular Gauss IFS. Then for all $x = \pi(\omega) \in \Lambda_E$ and $r > 0$, there exists $n$ such that
\begin{equation}
\label{GMFextra}
[\omega\given n + 1] \subset B(x,Cr)
\end{equation}
and
\begin{equation}
\label{GMF}
M(x,n,r) \leq \mu\big(B(x,r)\big) \lesssim M(x,n,Cr)
\end{equation}
where
\[
M(x,n,r) \df \sum_{\substack{a\in E \\ [(\omega\given n) a] \subset B(x,r)}} \mu((\omega\given n) a),
\]
and where $C\geq 1$ is a uniform constant.
\end{theorem}
We call this theorem a ``global measure formula'' due to its similarity to other global measure formulas found in the literature, e.g. \cite[Section 7]{Sullivan_entropy}, \cite[Theorem 2]{StratmannVelani}.
\begin{proof}
Note that the first inequality $M(x,n,r) \leq \mu(B(x,r))$ follows trivially from applying $\mu$ to both sides of the inclusion $\bigcup\{[\tau a] \subset B(x,r) : a\in E\} \subset B(x,r)$.

Given $x = \pi(\omega) \in \Lambda_E$ and $r > 0$, let $m\geq 0$ be maximal such that $B(x,r)\cap \Lambda_E \subset [\omega\given m]$. By applying the inverse transformation $\phi_{\omega\given m}^{-1}$ to the setup and using the bounded distortion property we may without loss of generality assume that $m = 0$, or equivalently that  $B(x,r)$ intersects at least two top-level cylinders. We now divide into two cases:
\begin{itemize}
\item If $[\omega_1] \subset B(x,r)$, then we claim that
\[
B(x,r)\cap \Lambda_E \subset \bigcup_{\substack{a\in E \\ [a] \subset B(x,Cr)}} [a]
\]
which guarantees \eqref{GMF} with $n=0$. Indeed, if $y = \pi(\tau) \in B(x,r)\cap \Lambda_E$, then $\frac1{\tau_1+1} \leq \pi(\tau) \leq \pi(\omega) + r \leq \frac1{\omega_1} + r$ and thus
\begin{align*}
\diam([\tau_1]) &\asymp \frac1{\tau_1^2} \lesssim \max(\frac1{\omega_1^2},r^2)\\
&\asymp \diam([\omega_1]) + r^2 \leq 2r + r^2 \lesssim r.
\end{align*}

\item If $[\omega_1]$ is not contained in $B(x,r)$, then one of the endpoints of $[\omega_1]$, namely $1/\omega_1$ or $1/(\omega_1+1)$, is contained in $B(x,r)$, but not both. Suppose that $1/\omega_1 \in B(x,r)$; the other case is similar. Now for all $N\in E$ such that $[(\omega_1 - 1)*1*N]\cap B(x,r)\neq \emptyset$ (where $*$ denotes concatenation), we have
\[
r\geq \dist(1/\omega_1,[(\omega_1-1)*1*N]) \asymp 1/(\omega_1^2 N) \asymp \dist(1/\omega_1,\min([\omega_1*N]))
\]
and thus $[\omega_1*N] \subset B(1/\omega_1,Cr) \subset B(x,(C+1)r)$ for an appropriately large constant $C$. Applying $\mu$ and summing over all such $N$ gives
\begin{align*}
\mu(B(x,r)) &\leq \mu([\omega_1]\cap B(x,r)) + \sum_{\substack{N\in E \\ [(\omega_1 - 1)*1*N]\cap B(x,r)\neq \smallemptyset}} \mu([(\omega_1 - 1)*1*N])\\
&\lesssim \sum_{\substack{N\in E \\ [\omega_1*N] \subset B(x,(C+1)r)}} \mu([\omega_1*N])
\end{align*}
which implies \eqref{GMF} with $n=1$. On the other hand, since
\[
r\geq \dist(x,1/\omega_1) \asymp 1/(\omega_1^2 \omega_2) \geq 1/(\omega_1^2 \omega_2^2) \asymp \diam([\omega\given 2]),
\]
we have $[\omega\given 2]\subset B(x,C r)$ as long as $C$ is sufficiently large.
\end{itemize}
\end{proof}

%\begin{example}
%Let $\II$ be the collection of all intervals removed in the construction of the middle-thirds Cantor set, i.e.
%\[
%\II = \{\phi_\omega([1/3,2/3]) : \omega\in E^*\},
%\]
%where $(\phi_a)_{a\in E}$ is the middle-thirds Cantor IFS.
%
%Fix $S \subset \N$. The \textsf{Cantor-interval IFS} corresponding to $S$ is the IFS generated by the template
%\[
%\{I\in\II : \ell(I) \in 3^{-S}\},
%\]
%i.e. the IFS consisting of all orientation-preserving contractions from $[0,1]$ onto an element of $\II$ whose length is $3^{-n}$ for some $n\in S$.
%
%If $S$ is infinite and has unbounded gaps, then the Cantor-interval IFS corresponding to $S$ does not satisfy the global measure formula.
%\end{example}
%\begin{proof}
%Let $(n_1,n_2)$ be a large gap, and let $B(x,r) = [0,1/4\cdot 3^{n_1}]$. [Continue\internal]
%\end{proof}

Fix $\origepsilon\in\{\pm1\}$,\Footnote{Loosely speaking, $\origepsilon=1$ when we are trying to prove results about Hausdorff measure, and $\origepsilon=-1$ when we are trying to prove results about packing measure.} a real number $\alpha > 0$, and a doubling dimension function $\psi(r) = r^\delta \Psi(r)$. We will assume that $\Psi$ is \textsf{$\origepsilon$-monotonic}, meaning that $\Psi$ is decreasing if $\origepsilon=1$ and increasing if $\origepsilon=-1$. Fix $\alpha > 0$, and for each $k\in\N$ let
\[
J_{k,\alpha,\origepsilon} \df \left\{a\in E : \exists r \in [\|\phi_a'\|,1] \;\; \text{with}\;\; r^{-\delta} \mu\big(B([a],r)\big) \lesseqgtr \alpha \Psi(\lambda^{-k} r) \right\}.
\]
Here $\lesseqgtr$ denotes $\geq$ if $\origepsilon=1$ and $\leq$ if $\origepsilon=-1$, and $\lambda > 1$ is as in Assumption \ref{assumptionC}. Write \
\[
S_{\alpha,\origepsilon} \df S_{\JJ_{\alpha,\origepsilon}} ~\text{for}~ \JJ_{\alpha,\origepsilon} \df \big(J_{k,\alpha,\origepsilon}\big)_{k=1}^\infty,
\]
as defined in Lemma \ref{lemmamuJk}. Note that $S_{\alpha,1}$ grows smaller as $\alpha$ grows larger, while $S_{\alpha,-1}$ grows larger as $\alpha$ grows larger.

\begin{proposition}
\label{propositionstars}
For all $\omega\in E^\N$,
\begin{align*}
\sup\{\alpha : \omega\in S_{\alpha,1}\} &\asymp \overline D_\mu^\psi(\pi(\omega))\\
\inf\{\alpha : \omega\in S_{\alpha,-1}\} &\asymp \underline D_\mu^\psi(\pi(\omega))\cdot
\end{align*}
\end{proposition}
\begin{proof}
Let $x = \pi(\omega)$, fix $r > 0$, and let $C$, $n$, and $\tau = \omega\given n$ be as in the global measure formula. Write $\tau\in W_k$ for some $k$, as in \eqref{Wkdef}. By the global measure formula, we have
\[
\sum_{\substack{a\in E \\ [\tau a] \subset B(x,r)}} \mu(\tau a)
\leq \mu\big(B(x,r)\big) \lesssim \sum_{\substack{a\in E \\ [\tau a] \subset B(x,Cr)}} \mu(\tau a).
\]
Now for each $\beta \geq 1$ let
\[
\Theta_\beta \df \frac{\mu(B(x,\beta r))}{\psi(\beta r)}\cdot
\]
Let $y = \pi(\sigma^n \omega)$, where $\sigma: E^\N \to E^\N$ is the shift map. Then there exist constants $C_2,C_3 > 0$ (independent of $x$, $r$, $n$, and $k$) such that for all $s > 0$,
\[
B(x,C_2 \lambda^{-k} s) \subset \phi_\tau(B(y,s)) \subset B(x,C_3 \lambda^{-k} s).
\]
Taking $s = C_3^{-1}\lambda^k r$ and $s = C_2^{-1} C \lambda^k \beta r$, and using the bounded distortion property and the fact that $\mu(\tau) \asymp \lambda^{-\delta k}$ yields
\begin{align*}
\Theta_1 &\lesssim \frac{1}{\psi(r)} \lambda^{-\delta k} \sum_{\substack{a\in E \\ [a] \subset B(y,C_2^{-1} C \lambda^k r)}} \mu(a)\\
\Theta_\beta &\gtrsim_\beta \frac{1}{\psi(r)} \lambda^{-\delta k} \sum_{\substack{a\in E \\ [a] \subset B(y,C_3^{-1} \lambda^k \beta r)}} \mu(a)
\end{align*}
Write $b = \omega_{n + 1}$, so that $x\in [\tau b] \subset B(x,C r)$ by \eqref{GMFextra} and thus by the bounded distortion property $y \in [b] \subset B(y,C_4 \lambda^k r)$ for sufficiently large $C_4$. Then $R \df 2C_4 \lambda^k r \geq \diam([b])$. Thus
\[
\mu\big(B(y,R)\big) \leq \mu\big(B([b],R)\big) \leq \mu\big(B(y,2R)\big),
\]
so $\Theta_1 \lesssim \Xi \lesssim_\beta \Theta_\beta$ for some $C_5,C_6 > 0$, where
\[
\Xi \df
\frac{1}{\psi(\lambda^{-k} R)} \lambda^{-\delta k} \sum_{\substack{a\in E \\ [a] \subset B([b],R)}} \mu(a) = \frac{1}{\Psi(\lambda^{-k} R)} R^{-\delta} \sum_{\substack{a\in E \\ [a] \subset B([b],R)}} \mu(a).
\]
Applying the global measure formula again yields
\[
\Theta_1 \lesssim \frac{1}{\Psi(\lambda^{-k} R)} R^{-\delta} \mu\big(B([b],R)\big) \lesssim \Theta_\beta
\]
for some $C_7,C_8 > 0$, and thus
\begin{align*}
\Theta_1 \geq C_9 \alpha \;\;\Rightarrow\;\;
R^{-\delta} \mu\big(B([b],R)\big) \geq \alpha \Psi(\lambda^{-k} R)
\;\;\Rightarrow\;\; \Theta_\beta \geq C_{10} \alpha\\
\Theta_\beta \leq C_{11} \alpha \;\;\Rightarrow\;\;
R^{-\delta} \mu\big(B([b],R)\big) \leq \alpha \Psi(\lambda^{-k} R)
\;\;\Rightarrow\;\; \Theta_1 \leq C_{12} \alpha
\end{align*}
for some $C_9,C_{10},C_{11},C_{12} > 0$ and for all $\alpha > 0$. It follows that
\begin{align*}
&\overline D_\mu^\psi(\pi(\omega)) \geq C_{13} \alpha \;\;\Rightarrow\;\;
\omega\in S_{\alpha,1}
\;\;\Rightarrow\;\; \overline D_\mu^\psi(\pi(\omega)) \geq C_{14} \alpha\\
&\underline D_\mu^\psi(\pi(\omega)) \leq C_{15} \alpha \;\;\Rightarrow\;\;
\omega\in S_{\alpha,-1}
\;\;\Rightarrow\;\; \underline D_\mu^\psi(\pi(\omega)) \leq C_{16} \alpha,
\end{align*}
since $\omega\in S_{\alpha,\origepsilon}$ if and only if there exist infinitely many $n,k,R$ such that $\omega\given n \in W_k$, $R\in [\|\phi_{\omega_{n+1}}'\|,1]$, and $R^{-\delta} \mu\big(B([b],R)\big) \lesseqgtr \alpha \Psi(\lambda^{-k} R)$, and $\limsup_{r\searrow 0} \Theta_1 = \limsup_{r\searrow 0} \Theta_\beta = \overline D_\mu^\psi(\pi(\omega))$ and $\liminf_{r\searrow 0} \Theta_1 = \liminf_{r\searrow 0} \Theta_\beta = \underline D_\mu^\psi(\pi(\omega))$. Taking the supremum (resp. infimum) with respect to $\alpha$ completes the proof.
\end{proof}

So to calculate $\HH^\psi(\mu)$ or $\PP^\psi(\mu)$, we need to determine whether the series $\Sigma_{\alpha,\origepsilon} \df \sum_{k = 1}^\infty \mu(J_{k,\alpha,\origepsilon})$ converges or diverges for each $\alpha > 0$:

\begin{lemma}
\label{lemmaETS}
Let $\origepsilon=1$, and suppose that $\Psi$ is $\origepsilon$-monotonic. If $\sum_{k = 1}^\infty \mu(J_{k,\alpha,\origepsilon})$ converges (resp. diverges) for all $\alpha > 0$, then $\HH^\psi(\mu)=\infty$ (resp. $=0$); otherwise $\HH^\psi(\mu)$ is positive and finite. If $\origepsilon=-1$, the analogous statement holds for $\PP^\psi(\mu)$.
\end{lemma}
\begin{proof}
By Corollary \ref{corollaryRTT} (and the subsequent remark), it suffices to show that $\overline D_\mu^\psi(\pi(\omega)) = 0$ (resp. $=\infty$) for a positive $\mu$-measure set of $\omega$s. By Proposition \ref{propositionstars}, this is equivalent to showing that $\sup\{\alpha : \omega\in S_{\alpha,1}\} = 0$ (resp. $=\infty$), or equivalently that $\omega\notin S_{\alpha,1}$ (resp. $\in S_{\alpha,1}$) for all $\alpha > 0$. For each $\alpha$, to show this for a positive $\mu$-measure set of $\omega$s it suffices to show that $\mu(S_{\alpha,1}) = 0$ (resp. $>0$), which by Lemma \ref{lemmamuJk} is equivalent to showing that $\sum_{k=1}^\infty \mu(J_{k,\alpha,1})$ converges (resp. diverges). The cases $\origepsilon = -1$ and where $\sum_{k=1}^\infty \mu(J_{k,\alpha,\origepsilon})$ converges for some $\alpha$ but diverges for others are proven similarly.
\end{proof}

\begin{lemma}
\label{lemmasigmaprime}
Assume that Assumptions \ref{assumptionPNT}, \ref{assumptionC}, and \ref{assumptionlyapunov} all hold, and that $\Psi$ is $\origepsilon$-monotonic. Then there exists a constant $C \geq 1$ such that for all $\alpha > 0$ and $\origepsilon \in \{\pm 1\}$, we have
\[
\Sigma_{C^{\origepsilon}\alpha,\origepsilon}' \lesssim_{\plus,\times}
\Sigma_{\alpha,\origepsilon} \lesssim_{\plus,\times} \Sigma_{C^{-\origepsilon}\alpha,\origepsilon}'
\]
where
\begin{align*}
\Sigma_{\alpha,-1}' &\df \sum_{a\in E} \mu(a) \max_{1 \leq x\leq a/3} \log\Big(1/\Psi^{-1}\big(\alpha^{-1} x^{-\delta} \#(B(a,x)\cap E)\big)\Big)\\
\Sigma_\alpha'' &\df \sum_{a\in E} \mu(a) \log\Big(1/\Psi^{-1}\big(\alpha^{-1} F(a^{-1})\big)\Big)\\
\Sigma_{\alpha,1}' &\df \Sigma_{\alpha,-1}' + \Sigma_\alpha''.
\end{align*}
Here $F(r) = r^\delta f(r^{-1})$, where $f$ is as in Assumption \ref{assumptionPNT}.
\end{lemma}
\begin{proof}
Indeed,
\begin{align*}
\sum_{k = 1}^\infty \mu(J_{k,\alpha,\origepsilon})
&= \sum_{a\in E} \mu(a) \#\big\{k\in\N : \exists r \in [\|\phi_a'\|,1] \;\;\; r^{-\delta} \mu(B([a],r)) \lesseqgtr \alpha \Psi(\lambda^{-k} r)\big\}\\
&= \sum_{a\in E} \mu(a) \max\big\{k\in\N : \exists r \in [\|\phi_a'\|,1] \;\;\; r^{-\delta} \mu(B([a],r)) \lesseqgtr \alpha \Psi(\lambda^{-k} r)\big\}\\
&\;\;\;\;(\text{since $\Psi$ is decreasing if $\origepsilon=1$ and increasing if $\origepsilon=-1$})\\
&\asymp_{\plus,\times} \sum_{a\in E} \mu(a) \max\Big(0,\max_{r \in [\|\phi_a'\|,1]} \log_\lambda\Big(r/\Psi^{-1}\big(\alpha^{-1} r^{-\delta} \mu(B([a],r))\big)\Big)\Big)\\
&\in \left[\sum_{a\in E} \mu(a) \log_\lambda\|\phi_a'\|,C\right] + \sum_{a\in E} \mu(a) \max_{r \in [\|\phi_a'\|,1]} \log_\lambda\Big(1/\Psi^{-1}\big(\alpha^{-1} r^{-\delta} \mu(B([a],r))\big)\Big)\\
&\;\;\;\;(\text{since $r \leq 1 \lesssim 1/\Psi^{-1}\big(\alpha^{-1} r^{-\delta} \mu(B([a],r))\big)$})
\end{align*}
The first term is finite by Assumption \ref{assumptionlyapunov}. The second term can be analyzed by considering
\begin{align*}
\Sigma_{1,\alpha} &\df \sum_{a\in E} \mu(a) \max_{a^{-1} \leq r \leq 1} \log_\lambda\Big(1/\Psi^{-1}\big(\alpha^{-1} r^{-\delta} \mu(B([a],r))\big)\Big)\\
\Sigma_{2,\alpha} &\df \sum_{a\in E} \mu(a) \max_{\|\phi_a'\|\leq r \leq a^{-1}/3} \log_\lambda\Big(1/\Psi^{-1}\big(\alpha^{-1} r^{-\delta} \mu(B([a],r))\big)\Big)
\end{align*}
Then
\[
\Sigma_{1,\alpha} + \Sigma_{2,\alpha} \lesssim_{\plus,\times} \Sigma_{\alpha,\origepsilon} \lesssim_{\plus,\times} \Sigma_{1,3^\delta\alpha} + \Sigma_{2,3^{-\delta}\alpha}.
\]
Now for $r > 0$ sufficiently small we have
\begin{align*}
\mu(B(0,r)) &\geq \sum_{a\geq r^{-1}} \mu(a) \asymp \sum_{a\geq r^{-1}} a^{-2\delta} \\
&\asymp \sum_{k = 0}^\infty \#(E\cap [2^k r^{-1}, 2^{k+1} r^{-1}]) (2^k r^{-1})^{-2\delta} \\
&\asymp \sum_{k = 0}^\infty f(2^k r^{-1}) 2^{-2k\delta} r^{2\delta} \asymp f(r^{-1}) r^{2\delta}
\;\;\;\;\text{ since $f$ is regularly varying with exponent $s<2\delta$}
\end{align*}
and similarly for the reverse direction, giving
\[
\mu(B(0,r)) \asymp r^{2\delta} f(r^{-1}).
\]
When $r\geq a^{-1}$, then $B(0,r)\cap [0,1] \subset B([a],r) \subset B(0,2r)$, so
\[
\mu(B([a],r)) \asymp r^{2\delta} f(r^{-1})
\]
and thus $\Sigma_{1,C^\origepsilon \alpha}' \leq \Sigma_{1,\alpha} \leq \Sigma_{1,C^{-\origepsilon} \alpha}'$, where
\[
\Sigma_{1,\alpha}' \df \sum_{a\in E} \mu(a) \max_{a^{-1} \leq r \leq 1} \log\Big(1/\Psi^{-1}\big(\alpha^{-1} r^\delta f(r^{-1})\big)\Big).
\]
Since $f$ is regularly varying with exponent $s>\delta$, the function $F(r) = r^\delta f(r^{-1})$ is monotonically decreasing for $r$ sufficiently small, while $\Psi^{-1}$ is decreasing (resp. increasing) if $\origepsilon=1$ (resp. $\origepsilon=-1$). It follows that the maximum occurs at $r = a^{-1}$ (resp. $r = 1$), corresponding to
\begin{equation}
\label{Hcomparison}
\Sigma_{1,\alpha}' \asymp_\plus \sum_{a\in E} \mu(a) \log\Big(1/\Psi^{-1}\big(\alpha^{-1} F(a^{-1})\big)\Big) \;\;\;\;\text{ if $\origepsilon = 1$}
\end{equation}
and
\begin{equation}
\label{Hcomparison2}
\Sigma_{1,\alpha}' \asymp_\plus \sum_{a\in E} \mu(a) \text{ const.} < \infty \;\;\;\;\text{ if $\origepsilon = -1$}
\end{equation}
The latter series always converges, whereas the former series may either converge or diverge.

On the other hand, we have
\[
\Sigma_{2,\alpha} \asymp \sum_{a\in E} \mu(a) \max_{a^{-2} \leq r\leq a^{-1}/3} \log\Big(1/\Psi^{-1}\big(\alpha^{-1} r^{-\delta} \mu(B([a],r))\big)\Big).
\]
Using the change of variables $r = a^{-2} x$ and the fact that $\mu(b) \asymp \mu(a) \asymp a^{-2\delta}$ for all $b\in E$ such that $B([a],r)\cap [b]\neq\emptyset$, we get $\Sigma_{2,C^\origepsilon\alpha}' \lesssim \Sigma_{2,\alpha} \lesssim \Sigma_{2,C^{-\origepsilon}\alpha}'$, where
\begin{equation}
\label{otherpart}
\Sigma_{2,\alpha}' \df \sum_{a\in E} \mu(a) \max_{1 \leq x\leq a/3} \log\Big(1/\Psi^{-1}\big(\alpha^{-1} x^{-\delta} \#(B(a,x)\cap E)\big)\Big)
\end{equation}
and $C \geq 1$ is a constant. Combining \eqref{Hcomparison}, \eqref{Hcomparison2}, and \eqref{otherpart} yields the conclusion.
\end{proof}

\section{Proofs of main theorems}
In this section we consider the Gauss IFS $\Phi_P$, where $P$ is the set of primes. We begin with a number-theoretic lemma.

\begin{lemma}
For all $\delta < 1$,
\begin{equation}
\label{hoheisel1}
\#(P\cap B(a,x)) \lesssim (x/a)^\delta f(a) \text{ for } 1 \leq x \leq a/3
\end{equation}
where $f(N) = N/\log(N)$ is as in Assumption \ref{assumptionPNT}.
\end{lemma}
\begin{proof}
A well-known result of  Hoheisel \cite{Hoheisel} states that there exists $\theta < 1$ such that
\begin{equation}
\label{hoheisel2}
\#(P\cap [a,b]) \asymp \frac{b - a}{\log(a)} \text{ if } a^\theta \leq b - a \leq a.
\end{equation}
(This result has seen numerous improvements, see \cite{Pintz2} for a survey, but it does not matter very much for our purposes, although the lower bound does make a difference in our upper bound for the exact packing dimension. The most recent improvements are $\theta = 6/11 + \epsilon$ for the upper bound \cite{LouYao2} and $\theta = 21/40$ for the lower bound \cite[p.562]{BHP}.)
% For the record, here is a (partial) list of improvements to \eqref{hoheisel2}:
% Hoheisel '30
% Tchudakoff '36
% Cramer '37 (\theta = 3/4)
% A. E. Ingham [Quart. J. Math. Oxford Ser. 8 (1937), 255–266]
% G. Halász [Acta Math. Acad. Sci. Hungar. 19 (1968), 365–403; MR0230694]
% H. L. Montgomery [Topics in multiplicative number theory, Lecture Notes in Math., Vol. 227, Springer, Berlin, 1971]
% M. N. Huxley [On the difference between consecutive primes]
% \cite{Heath-Brown4, LouYao}

% Lower bounds:
% D. R. Heath-Brown and H. Iwaniec [On the difference between consecutive primes]
% Iwaniec and Jutila
% Iwaniec and Pintz
% Mozzochi
% S. T. Lou and Q. Yao [Hardy-Ramanujan J. 15 (1992), 1–33 (1993); MR1215589]
% R. C. Baker and G. Harman [The difference between consecutive primes]

It follows that
\[
\#(P\cap B(a,x)) \lesssim \frac{x}{\log(a)} \text{ if } a^\theta \leq x \leq a/3.
\]
In this case, since $\delta < 1$ and $x \leq a$ we have
\[
\frac{x}{\log(a)} = \frac{x}{a} f(a) \leq \left(\frac{x}{a}\right)^\delta f(a)
\]
and combining yields \eqref{hoheisel1} in this case. On the other hand, if $1 \leq x \leq a^\theta$, then
\[
\#(P\cap B(a,x)) \leq 2x + 1 \lesssim (x/a)^\delta f(a),
\]
since
\[
x^{1 - \delta} \leq a^{(1 - \delta)\theta} \lesssim a^{1 - \delta}/\log(a),
\]
demonstrating \eqref{hoheisel1} for the second case.
%Note that a result of Maier \cite{Maier} (see also \cite[\63]{Soundararajan}) implies that the Cram\'er model gives wrong predictions in some circumstances. However, the deviations from the Cram\'er prediction exhibited there are much smaller (in a logarithmic sense) than what would be needed for Conjecture \ref{conjecturecramergranville} to fail.
\end{proof}

Thus for appropriate $C \geq 1$,
\begin{align*}
\Sigma_{\alpha,-1}' &\leq \sum_{a\in P} \mu(a) \log\Big(1/\Psi^{-1}\big(C\alpha^{-1} \max_{1 \leq x\leq a/3} x^{-\delta} (x/a)^\delta f(a)\big)\Big)\\
&= \sum_{a\in P} \mu(a) \log\Big(1/\Psi^{-1}\big(C\alpha^{-1} a^{-\delta} f(a)\big)\Big)\\
&= \sum_{a\in P} \mu(a) \log\Big(1/\Psi^{-1}\big(C\alpha^{-1} F(a^{-1})\big)\Big) \asymp_\plus \Sigma_{C^{-1}\alpha}''.
\end{align*}
It follows that $\Sigma_{\alpha,1}' \lesssim \Sigma_{C^{-1}\alpha}''$.

\begin{proof}[Proof of Theorem \ref{theoremHD}]
Set $\origepsilon=1$ and $E = P$. If $\Psi$ is increasing, then $\HH^\psi(\mu) \lesssim \HH^\delta(\mu) = 0$ and the series \eqref{HDseries} diverges, so we may henceforth assume that $\Psi$ is decreasing, allowing us to use Lemma \ref{lemmasigmaprime}. We have\Footnote{In this computation we use the Iverson bracket notation $[\Phi] = \begin{cases}1 & \text{$\Phi$ true}\\ 0 & \text{$\Phi$ false}\end{cases}$.}
\begin{align*}
\Sigma_\alpha'' &\asymp_\plus  \sum_{a \in P} \mu(a) \sum_{k=1}^\infty \left[ k \leq \log_\lambda \left( 1/\Psi^{-1}(\alpha^{-1} F(a^{-1}) \right) \right]\\
&= \sum_{a\in P} \mu(a) \sum_{k=1}^\infty \left[F^{-1}(\alpha\Psi(\lambda^{-k})) \geq a^{-1}\right]\\
&\asymp \sum_{k=1}^\infty \sum_{\substack{a\in P \\ a \geq 1/F^{-1}(\alpha\Psi(\lambda^{-k}))}} a^{-2\delta}\\
&\asymp \sum_{k=1}^\infty \frac{x^{1-2\delta}}{\log x}\given_{x=1/F^{-1}(\alpha\Psi(\lambda^{-k}))}\\
&\asymp \sum_{k=1}^\infty \frac{x^{1-2\delta}}{\log x}\given_{x=(y\log y)^{1/(1-\delta)},\; y=\alpha\Psi(\lambda^{-k})}\footnotemark\\
&\asymp \sum_{k=1}^\infty \frac{(y\log y)^{\frac{1-2\delta}{1-\delta}}}{\log y}\given_{y=\alpha\Psi(\lambda^{-k})}\\
&\asymp \sum_{k=1}^\infty \frac{(y\log y)^{\frac{1-2\delta}{1-\delta}}}{\log y}\given_{y=\Psi(\lambda^{-k})}.
\end{align*}
\footnotetext{We have $f(x) = x/\log x$, thus $F(r) = r^{\delta-1} / \log(r^{-1})$ and $F^{-1}(x) \asymp (x\log x)^{1/(\delta-1)}$.}
If this series converges for all $\alpha > 0$, then so do $\Sigma_{\alpha,1}'$ and (by Lemma \ref{lemmasigmaprime}) $\sum_k \mu(J_{k,\alpha,1})$, and thus by Lemma \ref{lemmaETS} we have $\HH^\psi(\mu)=\infty$. On the other hand, if the series diverges, then so does $\sum_k \mu(J_{k,\alpha,1})$ for all $\alpha > 0$, and thus by Lemma \ref{lemmaETS} we have $\HH^\psi(\mu)=0$.
\end{proof}

\begin{proof}[Proof of Theorem \ref{theoremPD1}]
We can get bounds on the exact packing dimension of $\mu_P$ by using known results about the distribution of primes. First, we state the strongest known lower bound on two-sided gaps:

\begin{theorem}[\cite{Maier2}]
Let $p_n$ denote the $n$th prime and let $d_n = p_{n+1} - p_n$. For all $k$,
\[
\limsup_{n\to\infty} \frac{\min(d_{n + 1},\ldots,d_{n + k})}{\phi(p_n)} > 0
\]
where $\phi$ is as in \eqref{phidef}.
\end{theorem}

Let $k = 2$ and let $(n_\ell)$ be a sequence along which the limsup is achieved, and let $a_\ell = p_{n_\ell + 2}\in P$ for some $\ell\in\N$, so that $P\cap B(a_\ell,x) = \{a_\ell\}$, where $x = c\phi(a_\ell)$ for some constant $c>0$. Then we have
\[
\Sigma_{\alpha,-1}' \gtrsim \sum_{a\in P} a^{-2\delta} \log\Big(1/\Psi^{-1}\big(\alpha^{-1} x^{-\delta} \#(P\cap B(a,x))\big)\Big)
\geq \sum_{\ell\in\N} a_\ell^{-2\delta} \log\Big(1/\Psi^{-1}\big(\alpha^{-1} c^{-\delta}\phi(a_\ell)^{-\delta} \big)\Big).
\]
Let $\psi(r) = r^\delta \phi^{-\delta}(\log(1/r))$ as in \eqref{Ppsimuinfty}, so that $\Psi(r) = \phi^{-\delta}(\log(1/r))$ and $\Psi^{-1}(x) = \exp(-\phi^{-1}(x^{-1/\delta}))$. Note that
\[
\phi^{-1}(x) = \log\big(1/\Psi^{-1}(x^{-\delta})\big)
\asymp e^x \frac{\log^2 x}{x \log\log x}
\]
and thus letting $\alpha = c^{-\delta}\gamma^{-\delta}$,
\begin{align*}
a^{-2\delta} \log\Big(1/\Psi^{-1}\big(\alpha^{-1} c^{-\delta}\phi(a)^{-\delta} \big)\Big)
&= a^{-2\delta} \phi^{-1}(\gamma \phi(a))\\
&\asymp a^{-2\delta} a^\gamma
\end{align*}
and thus $\Sigma_{\alpha,-1}'$ diverges for $\gamma > 2\delta$. Combining with Lemmas \ref{lemmaETS} and \ref{lemmasigmaprime} demonstrates \eqref{Ppsimuinfty}.

On the other hand, we use the lower bound of Hoheisel's theorem \eqref{hoheisel2} to get an upper bound on the exact packing dimension. Let $\theta = 21/40$. Fix $a\in P$ and $1 \leq x \leq a/3$, and let $\theta = 21/40$ as in \eqref{hoheisel1}. If $x \geq a^\theta$, then we have
\[
x^{-\delta} \#(P\cap B(a,x)) \gtrsim x^{-\delta} \frac{x}{\log(a)} \geq a^{-\theta\delta} \frac{a^\theta}{\log(a)} \gtrsim a^{-\theta\delta},
\]
and if $x \leq a^\theta$, then we have
\[
x^{-\delta} \#(P\cap B(a,x)) \geq x^{-\delta} \geq a^{-\theta\delta}.
\]
Thus for appropriate $c_2 > 0$
\[
\Sigma_{\alpha,-1}' \asymp \sum_{a\in P} a^{-2\delta} \max_{1 \leq x \leq a/3} \log\Big(1/\Psi^{-1}\big(\alpha^{-1} x^{-\delta} \#(P\cap B(a,x))\big)\Big)
\leq \sum_{a\in P} a^{-2\delta} \log\Big(1/\Psi^{-1}\big(c_2 \alpha^{-1} a^{-\theta\delta}\big)\Big).
\]
Letting
\[
\psi(r) = r^\delta \log^{-s}(1/r), \;\;
\Psi(r) = \log^{-s}(1/r), \;\;
\Psi^{-1}(x) = \exp(-x^{-1/s}),
\]
we get
\[
\Sigma_{\alpha,-1}' \lesssim \sum_{a\in P} a^{-2\delta} a^{\theta\delta/s} < \infty \text{ if } \theta\delta/s < 2\delta - 1
\]
Combining with Lemmas \ref{lemmaETS} and \ref{lemmasigmaprime} demonstrates \eqref{Ppsimu0}.
\end{proof}

\begin{proof}[Proof of Theorem \ref{theoremPD2}]
In what follows we assume the cases $k=1,2$ of Conjecture \ref{conjecturecramergranville}. From the case $k=1$ of Conjecture \ref{conjecturecramergranville}, in particular from $R_1 < \infty$, it follows that the gaps between primes have size $d_n = O(\log^2(p_n))$, and thus for all $a\in P$ and $1\leq x\leq a/3$, we have
\begin{equation}
\label{xdeltacomparison1}
x^{-\delta} \#(P\cap B(a,x)) \gtrsim x^{-\delta}\left( \frac{x}{\log^2(a)} + 1\right) \geq x^{-\delta}\left(\frac{x}{\log^2(a)}\right)^\delta = \frac{1}{\log^{2\delta}(a)}\cdot
\end{equation}
On the other hand, from the case $k=2$, in particular $R_2 > 0$, it follows that for an appropriate constant $c > 0$ there exists an infinite set $I \subset P$ (i.e. the set $\{p_{n+2} : \min(d_{n+1},d_{n+2}) \geq c\log^2(p_{n+2})\}$ for an appropriate constant $c > 0$) such that for all $a\in I$ and $1\leq x = x_a = c\log^2(a) \leq a/3$ such that $P\cap B(a,x) = \{a\}$ and thus
\begin{equation}
\label{xdeltacomparison2}
x^{-\delta} \#(P\cap B(a,x)) = x^{-\delta} \asymp \frac{1}{\log^{2\delta}(a)}\cdot
\end{equation}
Now let $\psi(r) = r^\delta \log^{-2\delta}\log(1/r)$ be as in \eqref{psidef}, so that $\Psi(r) = \log^{-2\delta}(1/r)$ and $\Psi^{-1}(x) = \exp\big(-\exp(x^{-1/2\delta})\big)$. In particular, $\Psi^{-1}$ is increasing. It follows that for appropriate $C_1 \geq 1 \geq C_2 > 0$,
\begin{align*}
\big[a\in I\big]\log\left(1/\Psi^{-1}\Big(\alpha^{-1} \frac{C_1}{\log^{2\delta}(a)}\Big)\right)
&\underset{\eqref{xdeltacomparison2}}{\leq} \max_{1\leq x\leq a/3} \log\left(1/\Psi^{-1}\Big(\alpha^{-1} x^{-\delta} \#(B(a,x)\cap P)\Big)\right)\\
&\underset{\eqref{xdeltacomparison1}}{\leq} \log\left(1/\Psi^{-1}\Big(\alpha^{-1} \frac{C_2}{\log^{2\delta}(a)}\Big)\right)
\end{align*}
and plugging into \eqref{otherpart} yields
\[
\sum_{a\in I} \mu(a) \log\left(1/\Psi^{-1}\Big(\alpha^{-1} \frac{C_1}{\log^{2\delta}(a)}\Big)\right) 
\leq \Sigma_{\alpha,-1}' \leq
\sum_{a\in P} \mu(a) \log\left(1/\Psi^{-1}\Big(\alpha^{-1} \frac{C_2}{\log^{2\delta}(a)}\Big)\right).
\]
We get
\[
\sum_{a\in I} a^{-2\delta} a^{(\alpha C_1^{-1})^{1/2\delta}}
\lesssim \Sigma_{\alpha,-1}' \lesssim
\sum_{a\in P} a^{-2\delta} a^{(\alpha C_2^{-1})^{1/2\delta}}.
\]
By choosing $\alpha > 0$ so that $(\alpha C_2^{-1})^{1/2\delta} < 2\delta - 1$, we get $\Sigma_{\alpha,-1}' < \infty$, and by choosing $\alpha > 0$ so that $(\alpha C_1^{-1})^{1/2\delta} > 2\delta$, we get $\Sigma_{\alpha,-1}' = \infty$. It follows then from Lemmas \ref{lemmaETS} and \ref{lemmasigmaprime} that $\PP^\psi(\mu)$ is positive and finite.
\end{proof}

%\begin{remark}
%It is worth asking whether there is any other way to take known number-theoretic results and translate them to information about the function
%\[
%f(a) = \min_{1\leq x\leq a/3} x^{-\delta} \#(P\cap B(a,x)).
%\]
%However, if $a = p_n$, then
%\[
%\max(d_1,\ldots,d_n)^{-\delta} \leq f(a) \leq \min(d_{n - 1},d_n)^{-\delta}.
%\]
%Thus, if $\phi$ is any doubling function then
%\[
%C_2(\phi) \leq \limsup_{a\to\infty} \frac{f^{-\delta}(a)}{\phi(a)} \leq C_1(\phi),
%\]
%where $C_k$ is defined as in Conjecture \ref{conjecturecramergranville} [Check inequality directions\internal]. So any improvement to the lower bound would yield the existence of one-sided gaps of large size, and any improvement to the upper bound would yield an upper bound on the size of large two-sided gaps. Neither of these is known, although FGKT provide a new lower bound on the size of one-sided gaps \cite{FGKT} (\cite{Maynard} has an independent proof).
%\end{remark}

\bibliographystyle{amsplain}

\bibliography{bibliography}

\end{document}